\ifdefined\llncs
\documentclass[envcountsame]{llncs}
\else
\documentclass[reqno]{amsart}
\usepackage[foot]{amsaddr}
\fi

\input macros.tex

\begin{document}

\title{MIX $\star$-autonomous quantales
  \\ 
  and the continuous weak
  order
} 

\ifdefined\llncs
\author{Maria Jo\~{a}o Gouveia\inst{1}\thanks{Partially supported by
    FCT under grant SFRH/BSAB/128039/2016} \and Luigi
  Santocanale\inst{2}\thanks{Partially supported by the TICAMORE
    project ANR-16-CE91-0002-01}}
\institute{
  \email{mjgouveia@fc.ul.pt} \\
  Universidade de Lisboa, 
  1749-016, Lisboa, Portugal \\
  \and
  \email{luigi.santocanale@lis-lab.fr} \\
  LIS, CNRS UMR 7020, Aix-Marseille Universit\'e, France}
\authorrunning{Gouveia and Santocanale}
\else
\author[M. J. Gouveia]{Maria Jo\~{a}o Gouveia$^1$}
\address{
  $^1$Faculdade de Ci\^{e}ncias, Universidade de
  Lisboa, Portugal
}
\email{mjgouveia@fc.ul.pt}
\thanks{$^1$Partially supported by FCT under grant
SFRH/BSAB/128039/2016}

\author[L. Santocanale]{Luigi Santocanale$^2$}
\address{
  $^2$LIS, CNRS UMR 7020, Aix-Marseille Universit\'e,
  France
}
\email{luigi.santocanale@lis-lab.fr}
\thanks{$^2$Partially supported by the TICAMORE
  project ANR-16-CE91-0002-01}
\fi

\maketitle

\begin{abstract}
  The set of permutations on a finite set can be given a lattice
structure (known as the weak Bruhat order). The lattice structure is
generalized to the set of words on a fixed alphabet
$\Sigma = \set{x,y,z,\ldots }$, where each letter has a fixed number of
occurrences (these lattices are known as multinomial lattices and, in
dimension $2$, as lattices of lattice paths).  By interpreting the
letters $x,y,z,\ldots $ as axes, these words can be interpreted as
discrete increasing paths on a grid of a $d$-dimensional cube, where
$d = \card(\Sigma)$.

We show in this paper how to extend this order to images of continuous
monotone paths 
from the unit interval to a $d$-dimensional cube.  The key tool used
to realize this construction is the quantale $\LjI$ of \jcont
functions from the unit interval to itself; the construction relies on
a few algebraic properties of this quantale: it is \staraut and it
satisfies the mix rule.

We begin developing a structural theory of these lattices by
characterizing join-irreducible elements, and by proving these
lattices are generated from their \jirr elements under infinite joins.

\end{abstract}

\section{Introduction}

Combinatorial objects (trees, permutations, discrete paths, \ldots )
are pervasive in mathematics and computer science; often these
combinatorial objects can be organised into some ordered collection in
such a way that the underlying order is a lattice.

Building on our previous work on lattices of binary trees (known as
Tamari lattices or \Associahedra) and lattices of permutations (known
as weak Bruhat orders or \Permutohedra) as well as on related
constructions \cite{ORDER-24-3,YAAMA-51-03,EJC,IJAC,STA0,STA,JEMS}, we
have been led to ask whether these constructions can still be
performed when the underlying combinatorial objects are replaced with
geometric ones.

More precisely we have investigated the following problem.
Multinomial lattices \cite{BB} generalize \Permutohedra in a natural
way.  Elements of a multinomial lattice are words on a finite totally
ordered alphabet $\Sigma = \{\,x,y,z\ldots \,\}$ with a fixed number
of occurrences of each letter. The order is obtained as the reflexive
and transitive closure of the binary relation $\covered$ defined by
$wabu \covered wbau$, whenever $a,b \in \Sigma$ and
$a < b$ (if we consider words with exactly one occurrence of each
letter, then we have a \Permutohedron).  Now these words can be given
a geometrical interpretation as discrete increasing paths in some
Euclidean cube of dimension $d = \card(\Sigma)$, so the \wo can be
thought of as a way of organising these paths into a lattice
structure. When $\Sigma$ contains only two letters, then these
lattices are also known as lattices of (lattice) paths \cite{pinzani}
and we did not hesitate in \cite{ORDER-24-3} to call the multinomial
lattices ``lattices of paths in higher dimensions''.  The question
that we raised is therefore whether the \wo can be extended from
discrete paths to continuous increasing paths.

We already presented at the conference TACL 2011 
the following result, positively answering this question.
\begin{prop}
  Let $d \geq 2$.  Images of increasing continuous paths from
  $\vec{0}$ to $\vec{1}$ in $\mathbb{R}^{d}$ can be given the
  structure of a lattice; moreover, all the \Permutohedra and all the
  multinomial lattices can be embedded into one of these lattices
  while respecting the dimension $d$.
\end{prop}
We called this lattice the \emph{\cwo}.
The proof of this result was complicated by the many computations
arising from the structure of the reals and from analysis. We recently
discovered a cleaner proof of the above statement where all these
computations are uniformly derived from a few algebraic properties.
The algebra we need to consider is the one of the quantale $\LjI$ of
\jc functions from the unit interval to itself. This is a \saq, see
\cite{barr79}, and moreover it satisfies the mix rule, see
\cite{cockettSeely}.
The construction of the \cwo is actually an instance of a general
construction of a lattice $\Ld{Q}$ from a \saq $Q$ satisfying the mix
rule. When $Q = \two$ (the two-element Boolean algebra) this
construction yields the usual \wBo; when $Q = \LjI$, this construction
yields the \cwo. Thus, the step we took is actually an instance of
moving to a different set of (non-commutative, in our case) truth
values, as notably suggested in \cite{lawvere}.
What we found extremely surprising is that many deep geometric notions
(continuous monotone path, maximal chains, \ldots ) might be
characterised via this simple move and using the algebra of quantales.

Let us state our first main result.  Let
$\langle Q, 1,\otimes,\opp{}\rangle$ be a \saq (or a residuated
lattice), denote by $0$ and $\oplus$ the dual monoidal operations.
$Q$ is not supposed to be commutative, but we assume that it is cyclic
($x^{\star} = x\lrimpl 0 = 0 \rlimpl x$, for each $x \in Q$) and that
it satisfies the MIX rule ($x \otimes y \leq x \oplus y$, for each
$x,y \in Q$).  Let $d \geq 2$,
$\cd := \{\,(i,j) \mid 1\leq i < j \leq d\,\}$ and consider the
product $Q^{\cd}$. Say that a tuple $f \in Q^{\cd}$ is \emph{closed}
if $f_{i,j} \otimes f_{j,k} \leq f_{i,k}$, and that it is \emph{open}
if $f_{i,k} \leq f_{i,j} \oplus f_{j,k}$; say that $f$ is
\emph{clopen} if it is closed and open.
\begin{thm}
  The set of clopen tuples of $Q^{\cd}$, with the pointwise
  ordering, is a lattice.
\end{thm}
The above lattice is the one we denoted $\Ld{Q}$. The second main
result we aim to present relates the algebraic setting to the analytic
one:
\begin{thm}
  Clopen tuples of $\LjI^{\cd}$ bijectively correspond to images of
  monotonically increasing continuous functions $\p : \I \rto \I^{d}$
  such that $\p(0) = \vec{0}$ and $\p(1) = \vec{1}$.
\end{thm}
The results presented in this paper undoubtedly have a mathematical
nature, yet our motivations for developing these ideas originate from
various researches in computer science that we recall next.

\emph{Directed homotopy} \cite{goubault2003} was developed to
understand behavioural equivalence of concurrent
processes. Monotonically increasing paths might be seen as behaviours
of distributed processes whose local state variable can only
increase. The relationship between directed homotopies and lattice
congruences (in lattices of lattice paths) was already pinpointed in
\cite{ORDER-24-3}. In that paper we did not push further these ideas,
mainly because the mathematical theory of a \cwo was not yet
available.

In \emph{discrete geometry} discrete paths (that is, words on the
alphabet $\set{x,y,\ldots }$) are used to approximate continuous
lines. In dimension $2$, Christoffel words \cite{reutenauer} are
well-established approximations of a straight segment from $(0,0)$ to
some point $(n,m)$.  The lattice theoretic nature of this kind of
approximation is apparent from the fact that Christoffel words can
equivalently be defined as images of the identity/diagonal via the
right/left adjoints to the canonical embedding of the binomial lattice
$\L(n,m)$ into the lattice $\LjI$.  For higher dimensions, there are
multiple proposals on how to approximate a straight segment, see for
example \cite{ANDRES2003,FeschetR06,BertheLabbe11,ProvencalVuillon}.
It is therefore tempting to give a lattice theoretic notion of
approximation by replacing the binomial lattices with the multinomial
lattices and the lattice $\LjI$ with the lattice $\LId$.
The structural theory of the lattices $\LId$ already identifies
difficulties in defining such a notion of approximation.  For
$d \geq 3$, the lattice $\LId$ is no longer the Dedekind-MacNeille
completion of the sublattice of discrete paths whose steps are on
rational points---this is the colimit of the canonical embeddings of the
multinomial lattices into $\LId$; defining approximations naively via
right/left adjoints of these canonical embeddings is bound to be
unsatisfying. This does not necessarily mean that we should discard
lattice theory as an approach to discrete geometry; for example, we
expect that notions of approximation that take into consideration the
degree of generation of $\LId$ from multinomial lattices will be more
robust.

\bigskip

The paper is organized as follows.  We recall in \secRef{notation} some
facts on \jcont (or \mcont) functions and adjoints.
\secRef{latticesFromQuantales} describes the construction of the
lattice $\Ld{Q}$, for an integer $d \geq 2$ a lattice and a \msaq $Q$.
In \secRef{quantaleLI} we show that the quantale $\LjI$ of continuous
functions from the unit interval to itself is a \msaq, thus giving
rise to a lattice $\Ld{\LjI}$ (we shall denote this lattice $\LId$,
to ease reading).
In the following sections we formally instantiate our geometrical
intuitions.  \secRef{paths} introduces the crucial notion of path and
discusses its equivalent characterizations.  In
\secRef{pathsDimensionTwo} we shows that paths in dimension $2$ are in
bijection with elements of the quantale $\LjI$.  In
\secRef{pathsDimensionMore} we argue that paths in higher dimensions
bijectively correspond to clopen tuples of the lattice $\LjI^{\cd}$.
In \secRef{structure} we discuss some structural properties of the
lattices $\LjI$. We add concluding remarks in the final section.

\section{Elementary facts on \jcont functions}
\label{sec:notation}

Throughout this paper, $\setof{d}$ shall denote the set
$\set{1,\ldots ,d}$ while we let
$\cd := \set{(i,j) \mid 1 \leq i < j \leq d}$.

Let $P$ and $Q$ be complete posets; a function $f : P \rTo Q$ is
\emph{\jcont} (\resp \emph{\mcont}) if
\begin{align}
  \label{eq:meetcontinuous}
  f(\bigvee X) & = \bigvee_{x \in X} f(x)\,, 
  \;\;\;\;(\text{\resp}\; f(\bigwedge X) = \bigwedge_{x \in X} f(x))\,,  
\end{align}
for every $X\subseteq P$ such that $\bigvee X$ (\resp $\bigwedge X$)
exists.  Recall that $\bot_{P} := \bigvee \emptyset$ (\resp
$\top_{P} := \bigwedge \emptyset$) is the least (\resp greatest)
element of $P$. Note that if $f$ is join-continuous (\resp \mcont)
then $f$ is monotone and $f(\bot_P)=\bot_Q$ (\resp
$f(\top_P)=\top_Q$).
Let $f$ be as above; a map $g : Q \rto P$ is \emph{\LADJ} to $f$ if
$g(q) \leq p$ holds if and only if $q \leq f(p)$ holds, for each
$p \in P$ and $q \in Q$; it is \emph{\RADJ} to $f$ if $f(p) \leq q$ is
equivalent to $p \leq g(q)$, for each $p \in P$ and $q \in Q$.  Notice
that there is at most one function $g$ that is \LADJ (\resp \RADJ) to
$f$; we write this relation by $g = \ladj{f}$ (\resp $g =
\radj{f}$). Clearly, when $f$ has a \RADJ, then
$f = \ladj{(\radj{g})}$, and a similar formula holds when $f$ has a
\LADJ.  We shall often use the following fact:
\begin{lemma}
  If $f : P \rto Q$ is monotone and $P$ and $Q$ are two complete
  posets, then the following are equivalent:
  \begin{enumerate}
  \item $f$ is \jcont (\resp \mcont),
  \item $f$ has a \RADJ (\resp \LADJ).
  \end{enumerate}
\end{lemma}
If $f$ is \jcont (\resp \mcont), then we have
\begin{align*}
  \radj{f}(q) & = \bigvee \set{p \in P \mid f(p) \leq q} \qquad
  (\,\text{\resp} \;\;\ladj{f}(q) = \bigwedge \set{p \in P \mid q \leq
    f(p)}\,)\,, \tag*{for each $q \in Q$.}
\end{align*}
Moreover, if $f$ is surjective, then these formulas can be
strengthened so to substitute inclusions with equalities:
\begin{align}
  \label{eq:adjsurjective}
  \radj{f}(q) & = \bigvee \set{p \in P \mid  f(p) = q}
  \qquad (\,\text{\resp}
  \;\;\ladj{f}(q) = \bigwedge \set{p \in P \mid  q = f(p)}\,)\,, \\
  &\mbox{\hspace{80mm}}\tag*{for each $q \in Q$.}
\end{align}
The set of monotone functions from $P$ to $Q$ can be ordered
point-wise: $f \leq g$ if $f(p) \leq g(p)$, for each $p \in
P$. Suppose now that $f$ and $g$ both have \RADJ{s}; let us argue that
$f \leq g$ implies $\radj{g} \leq \radj{f}$: for each $q \in Q$, the
relation $\radj{g}(q) \leq \radj{f}(q)$ is obtained by transposing
$f(\radj{g}(q)) \leq g(\radj{g}(q)) \leq q$, where the inclusion
$g(\radj{g}(q)) \leq q$ is the counit of the adjunction.  Similarly,
if $f$ and $g$ both have \LADJ{s}, then $f \leq g$ implies
$\ladj{g} \leq \ladj{f}$.

\section{Lattices from \msaq{s}}
\label{sec:latticesFromQuantales}

A \emph{\saq} is a tuple
$\qQ = \langle Q, 1, \otimes, 0,\oplus, \oppfun\rangle$ where $Q$ is
a complete lattice, $\otimes$ is a monoid operation on $Q$ that
distributes over arbitrary joins, $ \oppfun : Q^{op} \rto Q$ is an
order reversing involution of $Q$, and $(0,\oplus)$ is a second monoid
structure on $Q$ which is dual to $(1,\otimes)$. This means that
\begin{align*}
  0 & = \opp{1} \quad \tand \quad f \oplus g = \opp{(\opp{g} \otimes
    \opp{f})}\,.
\end{align*}
Last but not least, the following relation holds:
\begin{align*}
  f \otimes g & \leq h \quad \tiff\quad  f \leq h\oplus \opp{g}\quad  \tiff\quad  g \leq
  \opp{f}\oplus h\,.
\end{align*}
Let us mention that we could have also defined a \saq as a residuated
(bounded) lattice
$\langle Q,\bot,\vee,\top,\land,1,\otimes,\lrimpl,\rlimpl\rangle$ such
that $Q$ is complete and comes with a cyclic dualizing element
$0$. The latter condition means that, for each $x \in Q$,
$x \lrimpl 0 = 0 \rlimpl x$ and, letting $\opp{x} := x \lrimpl 0$,
$\oppopp{x} = x$.  This sort of algebraic structure is also called
\emph{(pseudo) \staraut lattice} or \emph{involutive residuated
  lattice}, see e.g. \cite{Paoli2005,Tsinakis2006,Emanovsky2008}.

\begin{example}
  \label{ex:Sugihara}
  Boolean algebras are the \saq{s} such that $\wedge = \otimes$ and
  $\vee = \oplus$. 
  For a further example
  consider the following structure on the
  ordered set $\set{-1 < 0 <1}$:
  \begin{align*}
    \begin{array}{r@{\;\;}|@{\;\;}rrr}
      \otimes&-1&0&1\\\hline
      -1 &-1&-1&-1\\
      0 &-1&0&1\\
      1 &-1&1&1
    \end{array}
    \qquad
    \begin{array}{r@{\;\;}|@{\;\;}rr@{\;\;\,}r}
      \oplus&-1&0&1\\\hline
      -1 &-1&-1&1\\
      0 &-1&0&1\\
      1 &1&1&1
    \end{array}
    \qquad
    \begin{array}{r@{\;\;}|@{\;\;}r}
       &\star\\\hline
      -1 &1\\
      0 &0\\
      1 &-1
    \end{array}
  \end{align*}
  Together with the lattice structure on the chain, this structure
  yields a \saq, known as the Sugihara monoid on the three-element
  chain, see e.g. \cite{GALATOS20122177}.
\end{example}

We presented in \cite{STA} several ways on how to generalize the
standard construction of the \Permutohedron (aka the weak Bruhat
order). We give here a new one.
Given a \saq $Q$, we consider the product
$\prodd{Q} := \prod_{1 \leq i < j \leq d} Q$. Observe that, as a
product, $\prodd{Q}$ has itself the structure of a quantale, the
structure being computed coordinate-wise.  We shall say that a tuple
$f = \langle f_{i,j} \mid 1 \leq i < j \leq d\rangle$ is \emph{closed}
(\resp \emph{open}) if
\begin{align*}
  f_{i,j} \otimes f_{j,k} & \leq f_{i,k}
  \qquad
  (\resp   \,f_{i,k} \leq f_{i,j} \oplus f_{j,k}\,)
  \,.
\end{align*}
Clearly, closed tuples are closed under arbitrary meets and open
tuples are closed under arbitrary joins.
Observe that $f$ is closed if and only if
$\opp{f} = \langle \opp{(f_{\sigma(j),\sigma(i)})} \mid 1 \leq i < j
\leq d\rangle$ is open, where for $i \in [d]$,
$\sigma(i) = d - i + 1$. Thus, the correspondence sending $f$ to
$\opp{f}$ is an anti-isomorphism of $\PrdQ$, sending closed tuples to
open ones, and vice versa.  We shall be interested in tuples
$f \in \PrdQ$ that are \emph{clopen}, that is, they are at the same
time closed and open.

For $(i,j) \in \cd$, a subdivision of the interval $[i,j]$ is a
sequence of the form
$i = \ell_{0} < \ell_{1} < \ldots \ell_{k -1}< \ell_{k} = j$ with
$\ell_{i} \in [d]$, for $i = 1,\ldots ,k$.
We shall use $S_{i,j}$ for the set of subdivisions of the interval
$[i,j]$.
As closed tuples are closed under arbitrary meets, for each
$f \in \prodd{Q}$ there exists a least tuple $\closure{f}$ such that
$f \leq \closure{f}$ and $\closure{f}$ is closed; this tuple is easily
computed as follows:
\begin{align*}
  \closure{f}_{i,j}
  & = \bigvee_{i < \ell_{1} < \ldots \ell_{k -1}< j \in S_{i,j}}
    f_{i,\ell_{1}}\otimes f_{\ell_{1},\ell_{2}}\otimes \ldots \otimes f_{\ell_{k-1},j}\,.
\end{align*}
Similarly and dually, if we set
\begin{align*}
  \interior{f}_{i,j} & := \bigwedge_{i < \ell_{1} < \ldots \ell_{k
      -1}< j \in S_{i,j}} f_{i,\ell_{1}}\oplus
  f_{\ell_{1},\ell_{2}}\oplus \ldots \oplus
  f_{\ell_{k-1},j}\,.
\end{align*}
then 
$\interior{f}$ is the greatest open tuple $\interior{f}$ below $f$.

\begin{proposition}
  \label{prop:meetsAndJoins}
  Suppose that, for each $f \in \PrdQ$,
  $\cl{\int{(\cl{f})}} = \int{(\cl{f})}$. Then, for each
  $f \in \PrdQ$, $\int{(\cl{\int{f}})} = (\cl{\int{f}})$ as well. The
  set of clopen tuples is then a lattice.
\end{proposition}
\begin{proof}
  The first statement is a consequence of the duality sending
  $f \in \PrdQ$ to $\opp{f} \in \PrdQ$. Now the relation
  $\cl{\int{(\cl{f})}} = \int{(\cl{f})}$ amounts to saying that the
  interior of any closed $f$ is again closed. The other relation
  amounts to saying that the closure of an open is open.  For a family
  $\set{f_{i} \mid i \in I}$, with each $f_{i}$ clopen, define then
  \begin{align*}
    \bigvee_{\Ld{Q}} \set{ f_{i} \mid i \in I } & :=
    \cl{\bigvee_{\PrdQ} \set{ f_{i} \mid i \in I } }\,, & 
    \bigwedge_{\Ld{Q}} \set{ f_{i} \mid i \in I } & :=
    \int{(\bigwedge_{\PrdQ} \set{ f_{i} \mid i \in I } )}\,,
  \end{align*}
  and remark that, by our assumptions, the expressions on the right of
  the equalities denote clopen tuples.  It is easily verified then
  these are the joins and meets, respectively, among clopen tuples. 
  \qed
\end{proof}
\begin{lemma}
  \label{lemma:Mixequivalent}
  Consider the following inequalities:
  \begin{align}
      (\alpha \oplus \beta) \otimes (\gamma \oplus \delta) &\leq
      \alpha \oplus (\beta \otimes \gamma) \oplus \delta
      \label{eq:distr} \\
      \label{eq:mix}
      \alpha \otimes \beta & \leq \alpha \oplus \beta\,.
    \end{align}
    Then \eqref{eq:distr} is valid and \eqref{eq:mix} is equivalent
    to $0 \leq 1$. 
\end{lemma}
\begin{toappendix}
  \begin{proof}
    Considering that
    $\alpha \oplus \beta = \oppopp{\alpha} \oplus \beta = \opp{\alpha}
    \lrimpl \beta$ and, similalry,
    $\gamma \oplus \delta = \gamma \rlimpl \opp{\delta}$, we derive
    \begin{align*}
      \opp{\alpha} \otimes (\alpha \oplus \beta) \otimes (\gamma
      \oplus \delta) \otimes \opp{\delta} & = \opp{\alpha} \otimes
      (\opp{\alpha} \lrimpl \beta) \otimes (\gamma \rlimpl
      \opp{\delta}) \otimes \opp{\delta} \leq \beta \otimes \gamma \,,
    \end{align*}
    uning the units of the residuated structure. Yet, the inequality
    so deduced is equivalent to \eqref{eq:distr} by transposing.

    If \eqref{eq:mix} holds, then $0 \leq 1$ is derived by
    instatiating in this equation both $\alpha$ to $0$ and $\beta$ to
    $1$.
    Conversely, suppose that $0 \leq 1$ and observe then that
    $0 \otimes 0 \leq 0 \otimes 1 = 0$.  Letting $\beta = \gamma = 0$
    in \eqref{eq:distr}, we derive \eqref{eq:mix} as follows:
    \begin{align*}
      \alpha \otimes \delta & = (\alpha \oplus 0) \otimes (0 \oplus
      \delta) \leq \alpha \oplus (0 \otimes 0) \oplus \delta \leq
      \alpha \oplus 0 \oplus \delta = \alpha \oplus \delta\,.
      \tag*{\qed}
    \end{align*}
  \end{proof}
\end{toappendix}

The inequation \eqref{eq:mix} is known as the \emph{mix rule}.  We say
that a \saq $\qQ$ is a \msaq if the mix rule holds in $\qQ$.
\begin{theorem}
  If $Q$ is a \msaq and $f \in \PrLQ$ is closed, then so is
  $\interior{f}$.  Consequently, the set of clopen tuples of $\PrLQ$
  is a lattice.
\end{theorem}
\begin{proof}
  Let $i,j,k \in [d]$ with $i < j < k$.  We need to show that
  \begin{align*}
    \interior{f}_{i,j} \otimes \interior{f}_{j,k} & \leq
    f_{i,\ell_{1}}\oplus \ldots \oplus f_{\ell_{n-1},k}
  \end{align*}
  whenever $i < \ell_{1} < \ldots \ell_{n -1}< k \in S_{i,k}$.  This
  is achieved as follows. Let $u \in \set{0,1,\ldots ,n-1}$ such that
  $j \in [\ell_{u}, \ell_{u +1})$ and put
  \begin{align*}
    \alpha & := f_{i,\ell_{1}} \oplus \ldots \oplus f_{\ell_{u}}\,, &
    \delta & := f_{\ell_{u+1}} \oplus \ldots \oplus f_{\ell_{n-1},k} \, \\
    \beta & := f_{\ell_{u},j} \, &
    \gamma & := f_{j,\ell_{u +1}}\,.
  \end{align*}
  We let in the above definition $f_{\ell_{u},j} := 0$ when
  $j = \ell_{u}$.  Then
  \begin{align*}
    \int{f}_{i,j} \otimes \int{f}_{j,k} & \leq (\alpha \oplus \beta)
    \otimes (\gamma \oplus \delta)\,, \tag*{by definition of
      $ \int{f}_{i,j}$ and $\int{f}_{j,k}$,}
    \\
    & \leq \alpha \oplus (\beta \otimes \gamma) \oplus \delta\,,
    \tag*{by the inequation~\eqref{eq:distr},}
    \\
    & \leq \alpha \oplus f_{\ell_{u},\ell_{u+1}} \oplus \delta\,,
    \tag*{since $f$ is closed,} \intertext{(or, when $j = \ell_{u}$,
      by using $\beta = 0 \leq 1$ and
      $\gamma = f_{\ell_{u},\ell_{u+1}}$)} & = f_{i,\ell_{1}}\oplus
    \ldots \oplus f_{\ell_{n-1},k}\,.
  \end{align*}
  The last statement of the theorem is an immediate consequence of
  Proposition~\ref{prop:meetsAndJoins}.
  \qed
\end{proof}

\begin{definition}
  For $\qQ$ a \msaq, $\Ld{\qQ}$ shall denote the lattice of clopen tuples
  of $\PrLQ$.
\end{definition}

\begin{example}
  Suppose $Q = \two$, the Boolean algebra with two elements
  $0,1$. Identify a tuple $\chi \in \prodd{\two}$ with the
  characteristic map of a subset $S_{\chi}$ of
  $\set{(i,j) \mid 1 \leq i < j \leq d}$. Think of this subset as a
  relation. Then $\chi$ is clopen if both $S_{\chi}$ and its
  complement in $\set{(i,j) \mid 1 \leq i < j \leq d}$ are transitive
  relations. These subsets are in bijection with permutations of the
  set $[d]$, see \cite{STA0}; the lattice $\Ld{\two}$ is therefore
  isomorphic to the well-known \Permutohedron, aka the \wBo.
  On the other hand, if $Q$ is the Sugihara monoid on the
  three-element chain described in Example~\ref{ex:Sugihara}, then the
  lattice of clopen tuples is isomorphic to the lattice of
  pseudo-permutations, see \cite{KLNPS01,STA}.
\end{example}
\begin{remark}
  For a fixed integer $d$ the definition of the lattice $\Ld{\qQ}$
  relies only on the algebraic structure of $Q$. This allows to say
  that the construction $\Ld{\,-\,}$ is functorial: if
  $f : \qQ_{0} \rto \qQ_{1}$ is a \saq homomorphism, then we shall
  have a lattice homomorphism
  $\Ld{f} : \Ld{\qQ_{0}} \rto \Ld{\qQ_{1}}$ (it might be also argued
  that if $f$ is injective, then so is $\Ld{f}$).  It also means that
  we can interpret the theories of the lattices $\Ld{\qQ}$ in the
  theory of the quantale $Q$. For example, if the equational theory of
  a quantale $\qQ$ is decidable, then 
  the equational theory of the lattice $\Ld{\qQ}$ is decidable as
  well.
\end{remark}

\section{The \msaq $\LjI$}
\label{sec:quantaleLI}

In this paper $\I$ shall denote the unit interval of the reals, that
is $\I := [0,1]$. We use $\LjI$ for the set of \jcont functions from
$\I$ to itself.  Notice that a monotone function $f : \I \rTo \I$ is
\jcont if and only if
\begin{align}
  \label{eq:joincontinuous}
  f(x) & = \bigvee_{y < x, \,y \in \I \cap \Q} f(y)\,, 
\end{align}
see Proposition~2.1, Chapter II of \cite{compendiumCL}.  As the
category of complete lattices and \jcont functions is a symmetric
monoidal closed category, for every complete lattice $L$ the set of
\jcont functions from $Q$ to itself is a monoid object in that
category, that is, a quantale, see
\cite{joyaltierney,rosenthal1990}. Thus, we have:
\begin{lemma}
  Composition induces a quantale structure on $\LjI$.
\end{lemma}
\begin{toappendix}
  Straightforward verification:
  \begin{align*}
    (f \circ \bigvee_{i \in I} g_{i})(x)
    & = f((\bigvee_{i \in I} g_{i})(x))
    = f(\bigvee_{i \in I} g_{i}(x))
    \\
    &
    = \bigvee_{i \in I} f(g_{i}(x))
    = \bigvee_{i \in I} ((f \circ g_{i})(x))
    = (\bigvee_{i \in I} f \circ g_{i})(x)\,, \\
    ((\bigvee_{i \in I} f_{i}) \circ g) (x)
    & = (\bigvee_{i \in I} f_{i})(g(x))
    = \bigvee_{i \in I} f_{i}(g(x)) 
    = (\bigvee_{i \in I} (f_{i} \circ g))(x)\,.
    \tag*{\qed}
  \end{align*}
\end{toappendix}
Let now $\LmI$ denote the collection of \mcont functions from
$\I$ to itself.
By duality, we obtain:
\begin{lemma}
  Composition induces a dual quantale structure on $\LmI$.
\end{lemma}
With the next set of observations we shall see $\LjI$ and $\LmI$ are
order isomorphic. For a monotone function $f : \I \rto \I$, define
\begin{align*}
  \meetof{f}(x) & = \bigwedge_{x < x'} f(x')\,, &
  \joinof{f}(x) & = \bigvee_{x' < x} f(x')\,.
\end{align*}
\begin{lemma}
  If $x < y$,
  then $\meetof{f}(x) \leq \joinof{f}(y)$. 
\end{lemma}
\begin{proof}
  Pick $z \in \I$ such that $x < z < y$ and observe then that
  $\meetof{f}(x) \leq f(z) \leq \joinof{f}(y)$.
  \qed
\end{proof}
\begin{proposition}
  \label{prop:meet-cont-closure}
  $\meetof{f}$ is the least \mc function above $f$ and $\joinof{f}$ is
  the greatest \jcont function below $f$.  The relations
  $\meetofjoinof{f} = \meetof{f}$ and $\joinofmeetof{f} = \joinof{f}$
  hold and, consequently, the operations
  $\joinof{(\,\cdot\,)} : \LmI \rto \LjI$ and
  $\meetof{(\,\cdot\,)} : \LjI \rto \LmI$ are inverse order preserving
  bijections.
\end{proposition}
\begin{proof}
  We prove only one statement. Let us show that $\meetof{f}$ is
  \mc; to this goal, we use
  equation~\eqref{eq:joincontinuous}:
  \begin{align*}
    \bigwedge_{x < t} \meetof{f}(t)
    & = \bigwedge_{x < t} \bigwedge_{t < t'} f(t')
    = \bigwedge_{x < t} f(t') 
    = \meetof{f}(x)\,.
  \end{align*}
  We observe next that $f \leq \meetof{f}$, as if $x < t$, then
  $f(x) \leq f(t)$.  This implies that if $g \in \LmI$ and
  $\meetof{f} \leq g$, then $f \leq \meetof{f} \leq g$.
  Conversely, if $g \in \LmI$ and $f \leq g$, then
  \begin{align*}
    \meetof{f}(x)  & = \bigwedge_{x < t} f(t)
    \leq  \bigwedge_{x < t} g(t) = g(x)\,.
  \end{align*}
  Let us prove the last sentence.  Clearly, both maps are order
  preserving. Let us show that $\meetofjoinof{f} = \meetof{f}$
  whenever $f$ is order preserving.
  We have $\meetofjoinof{f} \leq \meetof{f}$, since
  $\joinof{f} \leq f$ and $\meetof{(-)}$ is order preserves the
  pointwise ordering. For the converse inclusion, recall from the
  previous lemma that if $x < y$, then
  $\meetof{f}(x) \leq \joinof{f}(y)$, so
  \begin{align*}
    \meetof{f}(x) & \leq \bigwedge_{x < y} \joinof{f}(y) =  \meetofjoinof{f}(x)\,,
  \end{align*}
  for each $x \in \I$. Finally, to see that $\meetof{(-)}$ and
  $\joinof{(-)}$ are inverse to each other, observe that of
  $f \in \LmI$, then $\meetofjoinof{f} = \meetof{f} =f$.
  The equality $\joinofmeetof{f} =f$ for $f \in \LjI$ is derived similarly.

  \qed
\end{proof}

Recall that if $f \in \LjI$ (resp., $g \in \LmI$), then 
$\radj{f} \in \LmI$ (resp., $\ladj{f} \in \LjI$) denotes the right
adjoint of $f$ (resp., left adjoint of $g$). The following relation is
the key observation to uncover the \saq structure on $\LjI$.
\begin{lemma}
  \label{lemma:meetofjoinof}
  For each $f \in \LjI$, the
  relation $\joinof{(\radj{f})} = \ladj{(\meetof{f})}$ holds. 
\end{lemma}
\begin{proof}
  Let $f \in \LjI$; we shall argue that $x \leq \meetof{f}(y)$ if and
  only if $\joinof{(\radj{f})}(x) \leq y$, for each $x,y \in \I$.

  We begin by proving that $x \leq \meetof{f}(y)$ implies that
  $\joinof{(\radj{f})}(x) \leq y$.  Suppose $x \leq \meetof{f}(y)$ so,
  for each $z$ with $y < z$, we have $x \leq f(z)$. Suppose that
  $\joinof{(\radj{f})}(x) \not\leq y$, thus there exists $w < x$ such
  that $\radj{f}(w) \not\leq y$. Then $y < \radj{f}(w)$, so
  $x \leq f(\radj{f}(w)) \leq w$, contradicting $w < x$. Therefore,
  $\joinof{(\radj{f})}(x) \not\leq y$.

  Dually, we can argue that if $g \in \LmI$, then
  $\joinof{g}(x) \leq y$ implies $x \leq \meetof{(\ladj{g})}(y)$.
  Letting $g := \radj{f}$ in this statement we obtain the converse
  implication: $\joinof{(\radj{f})}(x) \leq y$ implies
  $x \leq \meetof{(\ladj{(\radj{f})})}(y) = \meetof{f}(y)$.  \qed
\end{proof}

For $f,g \in \LjI$, let us define
\begin{align*}
  f \otimes g & := g \circ f\,, & f \oplus g & := \joinof{(\meetof{g}
    \circ \meetof{f})}\,, & \opp{f} & = \joinof{(\radj{f})} =
  \ladj{(\meetof{f})} \,.
\end{align*}
\begin{proposition}
  The tuple $\langle \LjI,id,\otimes,id,\oplus,\oppfun{}\rangle$
  is a \msaq.
\end{proposition}
\begin{proof}
  The correspondence $\opp{(\cdot)}$ is order reversing as it is the
  composition of an order reversing function with a monotone one; by
  Lemma~\ref{lemma:meetofjoinof}, it is an involution:
  \begin{align*}
    \oppopp{f} & = \ladj{(\joinof{(\meetof{(\radj{f})})})}
   =\ladj{(\radj{f})} = f\,.
  \end{align*}
  To verify that
  \begin{align}
    \label{eq:duality}
    \opp{(f \otimes g)} & = \opp{g} \oplus \opp{f} 
  \end{align}
  holds, for any $f,g \in \LjI$, we compute as follows:
  \begin{align*}
    \opp{g} \oplus \opp{f} & = \joinof{(\meetof{(\opp{f})} \circ
      \meetof{(\opp{g})})} \\
    &= \joinof{(\meetof{\opprj{f}} \circ
      \meetof{\opprj{g}})}
    =  \joinof{(\radj{f} \circ \radj{g})}
    =  \opprj{(g \circ f)}
    = \opp{(g \circ f)} = \opp{(f \otimes g)}\,.
  \end{align*}
  We verify next that, for any $f,g,h \in \LjI$,
  \begin{align}
    \label{eq:residfirst}
    f \otimes g & \leq h 
    \quad\tiff\quad f \leq h \oplus \opp{g} \,.
  \end{align}
  Notice that
  $ h \oplus \opp{g} = \joinof{(\meetof{(\opp{g})} \circ \meetof{h})}
  = \joinof{(\meetof{\opprj{g}} \circ \meetof{h})} = \joinof{(\radj{g}
    \circ \meetof{h})}$,
  so
  \begin{align*}
    f \leq h \oplus \opp{g} \quad \tiff \quad & f \leq
    \joinof{(\radj{g} \circ \meetof{h})}\,,
    \tag*{by the equality just established,}
    \\
    \tiff \quad & f \leq\radj{g} \circ \meetof{h}\,,
        \tag*{by Proposition~\ref{prop:meet-cont-closure},}
        \\
        \tiff \quad & g \circ f  \leq \meetof{h}\,,
        \tag*{since $g(x) \leq h$ iff $x \leq \radj{g}(y)$,}
        \\
    \tiff \quad & f \otimes g = g \circ f \leq \joinof{\meetof{h}} =
    h\,,
     \tag*{using again Proposition~\ref{prop:meet-cont-closure}.}
   \end{align*}
   It is an immediate algebraic consequence of \eqref{eq:duality} and
   \eqref{eq:residfirst} that $f \otimes g \leq h$ is equivalent to
   $g \leq \opp{f} \oplus h$,
   for any $f,g,h \in \LjI$.
   Namely, we have
   \begin{align*}
     f \otimes g \leq h \quad & \tiff \quad f \leq h \oplus \opp{g}  \\
     & \tiff \quad \opp{(h \oplus \opp{g})} \leq \opp{f} \\
     & \tiff \quad g \otimes \opp{h} =  \oppopp{g} \otimes \opp{h}  \leq \opp{f} \\
     & \tiff \quad g \leq \opp{f} \oplus \oppopp{h} = \opp{f} \oplus
     h\,. 
   \end{align*}
  Finally, recall that the identity $id$ is both \jcont and \mc and
  therefore $\meetof{id} = id$. Then it is easily seen that $id$ is
  both a unit for $\otimes$ and for $\oplus$. As seen in
  Lemma~\ref{lemma:Mixequivalent}, this implies that $\LjI$ satisfies
  the mix rule.  \qed
\end{proof}

\section{Paths}
\label{sec:paths}

Let in the following $d \geq 2$ be a fixed integer; we shall use
$\I^{d}$ to denote the $d$-fold product of $\I$ with itself. That is,
$\I^{d}$ is the usual geometric cube in dimension $d$. Let us recall
that $\I^{d}$, as a product of the poset $\I$, has itself the
structure of a poset (the order being coordinate-wise) which,
moreover, is complete.

\begin{definition}
  \label{def:path}
  A \emph{path in $\I^{d}$} is a chain $C \subseteq \I^{d}$ with the
  following properties:
  \begin{enumerate}
  \item if $X \subseteq C$, then $\bigwedge X \in C$ and
    $\bigvee X \in C$,
  \item $C$ is dense as an ordered set: if $x,y \in C$ and $x < y$,
    then $x < z < y$ for some $z \in C$.
  \end{enumerate}
\end{definition}
We have given a working definition of the notion of path in $\I^{d}$,
as a totally ordered dense sub-complete-lattice of $\I^{d}$. The next
theorem state the equivalence among several properties, each of which
could be taken as a definition of the notion of path.
\begin{theorem}
  Let $d \geq 2$ and let $C \subseteq \I^{d}$. The following
  conditions are then equivalent:
  \begin{enumerate}
  \item $C$ is a path as defined in Definition~\ref{def:path};
  \item $C$ is a maximal chain of the poset $\I^{d}$;
  \item There exists a monotone (increasing) \tcont map
    $\p : \I \rto \I^{d}$ such that $\p(0) = \vec{0}$,
    $\p(1) = \vec{1}$, whose image is $C$.
  \end{enumerate}
\end{theorem}

\section{Paths in dimension 2}
\label{sec:pathsDimensionTwo}

We give next a further characterization of the notion of path, valid
in dimension $2$. The principal result of this Section,
Theorem~\ref{thm:pathsasjcont}, states that paths in dimension $2$ are
(up to isomorphism) just elements of the quantale $\LjI$.

\medskip

For a monotone function $f : \I \rto \I$ define
$C_{f} \subseteq \I^{2}$ by the formula
\begin{align}
  \label{def:CfDimTwo}
  C_{f} & := \bigcup_{x \in \I}\; \set{x}\times[\joinof{f}(x),\meetof{f}(x)]\,.
\end{align}
Notice that, by Proposition~\ref{prop:meet-cont-closure},
$C_{f} = C_{\joinof{f}} =C_{\meetof{f}}$.
\begin{proposition}
  \label{lemma:pathd2}
  $C_{f}$ is a path in $\I^{2}$.
\end{proposition}
\begin{proof}  
  We prove first that $C_{f}$, with the product ordering induced from
  $\I^{2}$, is a linear order. To this goal, we shall argue that, for
  $(x,y),(x',y') \in C_{f}$, we have $(x,y) < (x',y')$ iff either $x <
  x'$ or $x = x'$ and $y < y'$. That is, $C_{f}$ is a lexicographic
  product of linear orders, whence a linear order.
  Let us suppose that one of these two conditions holds: a) $x < x'$,
  b) $x = x'$ and $y < y'$.  If a), then
  $\meetof{f}(x) \leq \joinof{f}(x')$.  Considering that
  $y \in [\joinof{f}(x),\meetof{f}(x)]$ and
  $y' \in [\joinof{f}(x'),\meetof{f}(x')]$ we deduce $y \leq y'$. This
  proves that $(x,y) < (x',y')$ in the product ordering.  If b) then
  we also have $(x,y) < (x',y')$ in the product ordering.
  The converse implication, $(x,y) < (x',y')$ implies $x < x'$ or $x =
  x'$ and $y < y'$, trivially holds.

  We argue next that $C_{f}$ is closed under joins from $\I^{2}$.  Let
  $(x_{i},y_{i})$ be a collection of elements in $C_{f}$, we aim to
  show that $(\bigvee x_{i}, \bigvee y_{i}) \in C_{f}$, i.e.
  $\bigvee y_{i} \in
  [\joinof{f}(\bigvee x_{i}), \meetof{f}(\bigvee x_{i})]$.
  Clearly, as $y_{i} \leq \meetof{f}(x_{i})$, then
  $\bigvee y_{i} \leq \bigvee \meetof{f}(x_{i}) \leq
  \meetof{f}(\bigvee x_{i})$. Next, $\joinof{f}(x_{i}) \leq y_{i}$,
  whence
  $\joinof{f}(\bigvee x_{i}) = \bigvee \joinof{f}(x_{i}) \leq \bigvee
  y_{i}$.
  By a dual argument, we have that $(\bigwedge x_{i}, \bigwedge y_{i})
  \in C_{f}$.
  
  Finally, we show that $C_{f}$ is dense; to this goal let
  $(x,y),(x',y') \in C_{f}$ be such that $(x,y) < (x',y')$. If
  $x < x'$ then we can find a $z$ with $x < z < x'$; of course,
  $(z,f(z)) \in C_{f}$ and, but the previous characterisation of the
  order, $(x,y) < (z,f(z)) < (x',y')$ holds.  If $x = x'$ then
  $y < y'$ and we can find a $w$ with $y < w < y'$; as
  $ w \in [y,y'] \subseteq [\joinof{f}(x),\meetof{f}(x)]$, then
  $(x,w) \in C_{f}$; clearly, we have then
  $(x,y) < (x,w) < (x,y') = (x',y')$.  \qed
\end{proof}
For $C$ a path in $\I^{2}$, define
\begin{align}
  \label{eq:defvC}
  f_{C}^{-}(x) & := \bigwedge \set{y \mid (x,y) \in C}\,,
  &
  f_{C}^{+}(x) & := \bigvee \set{y \mid (x,y) \in C}\,.
\end{align}
Recall that a path $C \subseteq \I^{2}$ comes with \bcont surjective
projections $\pi_{1},\pi_{2} : C\rto \I$. Observe that
the following relations hold:
\begin{align}
  \label{eq:exprviaadjoints}
  f^{-}_{C} & = \pi_{2} \circ \ladj{(\pi_{1})}\,, & f^{+}_{C} & =
  \pi_{2} \circ \radj{(\pi_{1})}\,.
\end{align}
Indeed, we have
\begin{align*}
  \pi_{2}(\ladj{(\pi_{1})}(x))
  & = \pi_{2}(\bigwedge \set{(x',y) \in C \mid x = x' })
   \tag*{using  equation \eqref{eq:adjsurjective}}\,,
  \\
  & = \bigwedge \pi_{2}(\set{(x',y) \in C \mid x = x' }) = \bigwedge
  \set{ y \mid (x,y) \in C }\,.
\end{align*}
The other expression for $f^{+}$ is derived similarly. In particular,
the expressions in \eqref{eq:exprviaadjoints} show that
$f^{-} \in \LjI$ and $f^{+} \in \LmI$.

\begin{lemma}
  \label{lemma:firstInverse}
  We have
  \begin{align*}
    f_{C}^{-} & = \joinof{(f_{C}^{+})}\,, \quad f_{C}^{+}  =
    \meetof{(f_{C}^{-})}\,,
    \quad \text{and} \quad C = C_{f_{C}^{+}} = C_{f_{C}^{-}}\,.
  \end{align*}
\end{lemma}
\begin{proof}
  Let us firstly argue that $(x,y) \in C$ if and only if
  $f_{C}^{-}(x) \leq y \leq f_{C}^{+}(y)$. The direction from left to
  right is obvious. Conversely, it is easily verified that if
  $f_{C}^{-}(x) \leq y \leq f_{C}^{+}(y)$, then the pair $(x,y)$ is
  comparable with all the elements of $C$; then, since $C$ is a
  maximal chain, necessarily $(x,y) \in C$.

  Therefore, let us argue that $f_{C}^{+} = \meetof{(f_{C}^{-})}$; we
  do this by showing that $f_{C}^{+}$ is the least \mc function above
  $f_{C}^{-}$.  We have $f_{C}^{-}(x) \leq f_{C}^{+}(x)$ for each
  $x \in \I$ since the fiber sets
  $\pi_{1}^{-1}(x) = \set{(x',y) \in C \mid x' = x}$ are non empty.
  Suppose now that $f_{C}^{-} \leq g \in \LmI$. In order to prove that
  $f_{C}^{+} \leq g$ it will be enough to prove that
  $f_{C}^{+}(x) \leq g(x')$ whenever $x < x'$.  Observe that if
  $x < x'$ then $f_{C}^{+}(x) \leq f_{C}^{-}(x')$: this is because if
  $(x,y),(x',y') \in C$, then $x < x'$ and $C$ a chain imply
  $y \leq y'$. We deduce therefore $f_{C}^{+}(x) \leq f_{C}^{-}(x')
  \leq g(x')$.
  The relation $f_{C}^{-} = \joinof{(f_{C}^{+})}$ is proved similarly.
    \qed
\end{proof}

\begin{lemma}
  \label{lemma:secondInverse}
  Let $f : \I \rto \I$ be monotone and consider the path $C_{f}$.
  Then $\joinof{f} = f_{C_{f}}^{-}$ and $\meetof{f} = f_{C_{f}}^{+}$.
\end{lemma}
\begin{proof}
  For a monotone $f : \I \rto \I$ define $f' : \I \rto C_{f}$ by
  $f' := \langle id_{\I},f\rangle$, so $f = \pi_{2} \circ f'$.  Recall
  that $f^{-}_{C_{f}} = \pi_{2} \circ \ladj{(\pi_{1})}$. Therefore, in
  order to prove the relation
  $\joinof{f} = f^{-}_{C_{f}} = \pi_{2} \circ \ladj{(\pi_{1})}$ it
  shall be enough to prove that $\langle id, \joinof{f}\rangle$ is
  left adjoint to the first projection (that is, we prove that
  $\langle id, \joinof{f}\rangle = \ladj{(\pi_{1})}$, from which it
  follows that
  $\joinof{f} = \pi_{1}\circ \langle id, \joinof{f}\rangle = \pi_{2}
  \circ \ladj{(\pi_{1})}$).  This amounts to verify that, for
  $x \in \I$ and $(x',y) \in C_{f}$ we have $x \leq \pi_{1}(x',y)$ if
  and only if $(x,\joinof{f}(x)) \leq (x',y)$. To achieve this goal,
  the only non trivial observation is that if $x \leq x'$, then
  $\joinof{f}(x) \leq \joinof{f}(x') \leq y$.
  The relation $\meetof{f} =
  \pi_{2} \circ \radj{(\pi_{1})}$ is proved similarly.
  \qed
\end{proof}

\begin{theorem}
  \label{thm:pathsasjcont}
  There is a bijective correspondence between the following data:
  \begin{enumerate}
  \item paths in $\I^{2}$,
  \item \jc functions in $\LjI$,
  \item \mc functions in $\LmI$.
  \end{enumerate}
\end{theorem}
\begin{proof}
  According to Lemmas~\ref{lemma:firstInverse} and
  \ref{lemma:secondInverse}, the correspondence sending a path $C$ to
  $f_{C}^{-} \in \LjI$ has the mapping sending $f$ to $C_{f}$ as an
  inverse. Similarly, the correspondence
  $C \mapsto f_{C}^{+} \in \LmI$ has $f \mapsto C_{f}$ as inverse.
  \qed
\end{proof}

\section{Paths in higher dimensions}
\label{sec:weakBruhat}
\label{sec:pathsDimensionMore}

We show in this Section that paths in dimension $d$, as defined in
Section~\ref{sec:paths}, are in bijective correspondence with
clopen tuples of $\PrLI$, as defined in
Section~\ref{sec:latticesFromQuantales}; therefore, as established in
that Section, there is a lattice $\Ld{\LjI}$ whose underlying set can
be identified with the set of paths in dimension $d$.

\medskip

Let $f \in \PrLI$, so $f = \set{ f_{i,j} \mid 1 \leq i < j \leq
  d}$. We define then, for $1 \leq i < j \leq d$,
\begin{align*}
  f_{j,i} & := \opp{(\,f_{i,j}\,)} = \joinof{(\radj{(f_{i,j})})} \,.
\end{align*}
Moreover, for $i \in [d]$, we let $f_{i,i} := id$.

\begin{definition}
  We say that a tuple $f \in \PrLI$ is \emph{compatible} if
  $f_{j,k} \circ f_{i,j}  \leq f_{i,k}$,
  for each triple of elements $i,j,k \in [d]$.
\end{definition}
\begin{lemma}
  \label{lemma:enrichment}
  A tuple is compatible if and only if it is clopen.
\end{lemma}
\begin{proof}
  For $i < j < k$, compatibility yields
  $f_{i,j}\otimes f_{j,k} \leq f_{i,k}$ (closedness) and
  $f_{k,j}\otimes f_{j,i} \leq f_{k,i}$ which in turn is equivalent to
  $f_{i,k} \leq f_{i,j} \oplus f_{j,k}$ (openness).

  Conversely, suppose that $f$ is clopen. Say that the pattern $(ijk)$
  is satisfied by $f$ if $f_{i,j}\otimes f_{i,j} \leq f_{i,k}$. If
  $\card(\set{i,j,k}) \leq 2$, then $f$ satisfies the pattern $(ijk)$
  if $i=j$ or $j = k$, since then $f_{i,j} = id$ or $f_{j,k} = id$. If
  $i = k$, then $f_{i,j} \otimes f_{j,i} \leq id$ is equivalent to
  $f_{i,j} \leq id \oplus f_{i,j}$. Suppose therefore that
  $\card(\set{i,j,k}) = 3$.

  By assumption, $f$ satisfies $(ijk)$ and $(kji)$ whenever
  $i < j < k$.  Then it is possible to argue that all the patterns on
  the set $\set{i,j,k}$ are satisfied by observing that if $(ijk)$ is
  satisfied, then $(jki)$ is satisfied as well: from
  $f_{i,j}\otimes f_{j,k} \leq f_{i,k}$, derive
  $f_{i,j}\leq f_{i,k}\oplus f_{k,j}$ and then
  $f_{j,k} \otimes f_{k,i} \leq f_{j,i}$.
  \qed
\end{proof}
\begin{remark}
  \label{rem:suffCondClopen}
  Let $f \in \PrLI$ and suppose that, for some $i,j,k \in \setof{d}$,
  with $i < j < k$, $f_{i,k} = f_{i,k} \circ f_{i,j}$. That is, we
  have $f_{i,k} = f_{i,j} \otimes f_{j,k}$ and, using the mix rule, we
  derive $f_{i,k} \leq f_{i,j} \oplus f_{j,k}$.  Dually, a relation of
  the form
  $\MeetOf{f_{i,k}} = \MeetOf{f_{j,k}} \circ \MeetOf{f_{i,j}}$ is
  equivalent to $f_{i,k} = f_{i,j} \oplus f_{j,k}$ and implies
  $f_{i,j} \otimes f_{j,k} \leq f_{i,k}$.
\end{remark}
\begin{remark}
  Lemma~\ref{lemma:enrichment} shows that a clopen tuple of $\PrLI$
  can be extended in a unique way to a \emph{skew} enrichment of the
  set $[n]$ over $\LjI$, see \cite{lawvere,stubbe}. Dually, a clopen
  tuple gives rise to a unique skew metric on the set $[n]$ with
  values in $\LjI$. For a skew enrichment (or metric) we mean, here,
  that the law $f_{j,i} = f_{i,j}^{\star}$ holds; this law, which
  replaces the more usual requirement that a metric is symmetric, has
  been considered e.g. in \cite{pouzet}.
\end{remark}

If $C \subseteq \I^{d}$ is a path, then we shall use
$\pi_{i} : C \rto \I$ to denote the projection onto the $i$-th
coordinate. Then
$\pi_{i,j} := \langle \pi_{i},\pi_{j}\rangle : C \rto \I \times \I$.
\begin{definition}
  \label{def:defvCij}
  For a path $C$ in $\I^{d}$, let us define $v(C) \in \PrLI[d]$ by the
  formula:
  \begin{align}
    \label{eq:defvCij}
    v(C)_{i,j} & := \pi_{j}\circ \ladj{(\pi_{i})}\,, \quad (i,j) \in \couples{d}. 
  \end{align}
\end{definition}
\begin{remark}
  An explicit formula for $v(C)_{i,j}(x)$ is as follows:
  \begin{align}
    \label{eq:explicitvC}
    v(C)_{i,j}(x)
    & = \bigwedge \set{\pi_{j}(y) \in C\mid \pi_{i}(y) = x}\,.
  \end{align}
  Let $C_{i,j}$ be the image of $C$ via the projection
  $\pi_{i,j}$. Then $C_{i,j}$ is a path, since it is the image of a
  bi-continuous function from $\I$ to $ \I \times \I$.  Some simple
  diagram chasing (or the formula in~\eqref{eq:explicitvC}) shows that
  $v(C)_{i,j} = f^{-}_{C_{i,j}}$ as defined in \eqref{eq:defvC}.
\end{remark}

\begin{definition}
  For a compatible $f \in
  \PrLI$, 
  define
  \begin{align*}
    C_{f} & := \set{(x_{1},\ldots ,x_{d}) \mid f_{i,j}(x_{i}) \leq
      x_{j}, \text{ for all } i,j \in [d] }\,.
  \end{align*}
\end{definition}
\begin{remark}
  \label{remark:adj}
  Notice that the condition $f_{i,j}(x) \leq y$ is equivalent (by
  definition of $f_{i,j}$ or $f_{j,i}$) to the condition
  $x \leq \Meetof{f_{j,i}}(y)$.  Thus, there are in principle many
  different ways to define $C_{f}$; in particular, when $d = 2$ (so
  any tuple $\PrLI$ is compatible), the definition given above is
  equivalent to the one given in \eqref{def:CfDimTwo}.
\end{remark}

\begin{proposition}
  $C_{f}$ is a path.
\end{proposition}
The proposition is an immediate consequence of the following
Lemmas~\ref{lemma:Cftotal}, \ref{lemma:Cfcomplete} and
\ref{lemma:Cfdense}.

\begin{lemma}
  \label{lemma:Cftotal}
  $C_{f}$ is a total order.
\end{lemma}
\begin{proof}
  Let $x,y \in C_{f}$ and suppose that $x\not\leq y$, so there exists
  $i \in [d]$ such that $x_{i} \not\leq y_{i}$. W.l.o.g. we can
  suppose that $i = 1$, so $y_{1} < x_{1}$ and then, for $i > 1$, we have 
  $\Meetof{f_{1,i}}(y_{1}) \leq f_{1,i}(x_{1})$, whence $y_{i} \leq
  \Meetof{f_{1,i}}(y_{1}) \leq f_{1,i}(x_{1}) \leq x_{1}$.
  This shows that $y < x$.
  \qed
\end{proof}

\begin{lemma}
  \label{lemma:Cfcomplete}
  $C_{f}$ is closed under arbitrary meets and joins.
\end{lemma}
\begin{proof}
  Let $\set{x^{\ell} \mid \ell \in I}$ be a family of tuples in
  $C_{f}$.  For all $i,j \in \setof{d}$ and $\ell \in I$, we have
  $f_{i,j}(\bigwedge_{\ell \in I} x^{\ell}_{i})
   \leq f_{i,j}(x^{\ell}_{i}) \leq
   x^{\ell}_{j} $,
  and therefore
  $f_{i,j}(\bigwedge_{\ell \in I} x^{\ell}_{i}) \leq \bigwedge_{\ell
    \in I}x^{\ell}_{j}$. Since meets in $\I^{d}$ are computed
  coordinate-wise, this shows that $C_{f}$ is closed under arbitrary
  meets.  Similarly,
  $ f_{i,j}(x^{\ell}_{i}) \leq \bigvee_{\ell \in I}x^{\ell}_{j}$ and
  \begin{align*}
    f_{i,j}(\bigvee_{\ell \in I} x^{\ell}_{i}) & = \bigvee_{\ell \in
      I} f_{i,j}(x^{\ell}_{i}) \leq \bigvee_{\ell \in I}x^{\ell}_{j}
    \,,
  \end{align*}
  so $C_{f}$ is also closed under arbitrary joins.
  \qed
\end{proof}
\begin{lemma}
  \label{lemma:Crightadj}
  Let $f \in \PrLI$ be compatible.  Let $i_{0} \in \setof{d}$ and
  $x_{0} \in \I$; define $x \in \I^{d}$ by setting
  $x_{i} := f_{i_{0},i}(x_{0})$ for each $i \in \setof{d}$.
  Then $x \in C_{f}$ and $x = \bigwedge \set{y \in C_{f} \mid
    \pi_{i_{0}}(y) = x_{0} }$.
\end{lemma}
\begin{proof}
  Since $f$ is compatible, $ f_{i,j} \circ f_{i_{0},i} \leq
  f_{i_{0},j}$, for each $i, j \in\setof{d}$, so
  \begin{align*}
     f_{i,j}(x_{i}) & = f_{i,j}(f_{i_{0},i}(x_{0})) \leq 
     f_{i_{0},j}(x_{0})  = x_{j}\,.
  \end{align*}
  Therefore, $x \in C_{f}$. Observe that since $f_{i_{0},i_{0}} =
  id$, we have $x_{i_{0}} = x_{0}$ and
  $x$ so defined is such that $\pi_{i_{0}}(x) = x_{0}$.
  On the other hand, if $y \in C_{f}$ and $x_{0} \leq \pi_{i_{0}}(y)
  =y_{i_{0}}$, then $x_{i} = f_{i,i_{0}}(x_{0}) \leq
  f_{i,i_{0}}(y_{i_{0}}) \leq y_{i}$, for all $i
  \in\setof{d}$.  Thus $x = \bigwedge \set{y \in C_{f} \mid
    \pi_{i_{0}}(y) = x_{0} }$. \qed
\end{proof}
\begin{lemma}
  \label{lemma:Cfdense}
  $C_{f}$ is dense.
\end{lemma}
\begin{proof}
  Let $x,y \in C_{f}$ and suppose that $x < y$, so there exists $i_{0}
  \in \setof{d}$ such that $x_{i_{0}}< y_{i_{0}}$. Pick $z_{0} \in
  \I$ such that $x_{i_{0}} < z_{0} < y_{i_{0}}$ and define
  $z \in  C_{f}$ as in  Lemma~\ref{lemma:Crightadj}, 
  $z_{i} := f_{i_{0},i}(z_{0})$, for
  all $i \in \setof{d}$.
  We claim that $x_{i} \leq z_{i} \leq y_{i}$, for each $i \in
  \setof{d}$. From this and $x_{i_{0}} < z_{i_{0}} <
  y_{0}$ it follows that $x < z < y$. Indeed, we have
  $z_{i}  = f_{i_{0},i}(z_{0}) \leq f_{i_{0},i}(y_{i_{0}}) \leq
    y_{i}$.
    Moreover, $x_{i_{0}} <
    z_{0}$ implies $\Meetof{f_{i_{0},i}}(x_{i_{0}}) \leq
    f_{i_{0},i}(z_{0})$; by Remark~\ref{remark:adj}, we have $x_{i }
    \leq
    \Meetof{f_{i_{0},i}}(x_{i_{0}})$. Therefore, we also have $x_{i}
    \leq \Meetof{f_{i_{0},i}}(x_{i_{0}}) \leq f_{i_{0},i}(z_{0}) =
    z_{i}$.
    \qed
\end{proof}

\begin{lemma}
  \label{lemma:vCfeqf}
  If $f \in \PrLI$ is compatible, then $v(C_{f}) = f$.
\end{lemma}
\begin{proof}
  By Lemma~\ref{lemma:Crightadj}, the correspondence sending
  $x$ to $(f_{i,1}(x),\ldots
  ,f_{d,1}(x))$ is left adjoint to the projection $\pi_{i} : C_{f}
  \rto
  \I$. 
  In turn, this gives that $v(C_{f})_{i,j}(x) =
  \pi_{j}(\ladj{(\pi_{i})}(x)) = f_{i,j}(x)$, for any $i,j \in
  \setof{d}$. It follows that $v(C_{f}) = f$.  \qed
\end{proof}
\begin{lemma}
  \label{lemma:CvCeqC}
  For $C$ a path in $\I^{d}$,  we have $C_{v(C)} = C$.
\end{lemma}
\begin{proof}
  Let us show that $C \subseteq C_{v(C)}$. Let $c \in C$;  notice that
  for each $i, j \in\setof{d}$, we have 
  \begin{align*}
    v(C)_{i,j}(c_{i}) & = \pi_{j}(\ladj{(\pi_{i})}(c_{i}))
    = \pi_{j}(\ladj{(\pi_{i})}(\pi_{i}(c)) \leq \pi_{j}(c) = c_{j}\,,
  \end{align*}
  so $c \in C_{v(C)}$. For the converse inclusion, notice that
  $C \subseteq C_{v(C)}$ implies $C = C_{v(C)}$,  since every path is
  a maximal chain.
  \qed
\end{proof}

Putting together Lemmas~\ref{lemma:vCfeqf} and \ref{lemma:CvCeqC} we
obtain:
\begin{theorem}
  The correspondences, sending a path $C$ in $\I^{d}$ to the tuple
  $v(C)$, and a compatible tuple $f$ to the path $C_{f}$, are inverse
  bijections.
\end{theorem}

\section{Structure of the lattices $\LId$}
\label{sec:structure}

As final remarks, we present and discuss some structural properties of
the lattices $\LId$.

Recall that an element $p$ of a lattice $L$ is \emph{\jp} if, for any
\emph{finite} family $\set{x_{i} \mid i \in I}$,
$p \leq \bigvee_{i \in I} x_{i}$ implies $p \leq x_{i}$, for some
$i \in I$. A \emph{\cjp} element is defined similarly, by considering
arbitrary families in place of finite ones.  An element $p$ of a
lattice $L$ is \emph{\jirr} if, for any \emph{finite} family
$\set{x_{i} \mid i \in I}$, $p = \bigvee_{i \in I} x_{i}$ implies
$p = x_{i}$, for some $i \in I$; \emph{\cjirr} elements are defined
similarly, by considering arbitrary families.  If $p$ is \jp, then it
is also \jirr, and the two notions coincide on distributive lattices.

\paragraph{Join-prime elements of $\LjI$.}
We begin by describing the \jp elements of $\LjI$; this lattice being
distributive, \jp and \jirr elements coincide. For $x,y \in \I$, let
us put
\begin{align*}
  \ji{x,y}(t)
  & :=
  \begin{cases}
    \,0\,, & 0\leq t \leq x\,, \\
    \,y\,, &  x < t \,,
  \end{cases}
  &
  \JI{x,y}(t)
  & :=
  \begin{cases}
    \,0\,, & 0\leq t < x\,, \\
    \,y\,, & x \leq t < 1 \,, \\
    \, 1\,, & t = 1\,.
  \end{cases}
\end{align*}
so $\ji{x,y} \in \LjI$, $\JI{x,y} \in \LmI$ and $\JI{x,y} =
\Ji{x,y}$.
We call a function of the form $\ji{x,y}$  a \emph{\osf}.
Notice that if $x = 1$ or $y = 0$, then $\ji{x,y}$ is the constant
function that takes $0$ as its unique value; said otherwise,
$\ji{x,y} = \bot$.  We say that $\ji{x,y}$ is a \emph{prime \osf} if
$x < 1$ and $0 < y$; we say that $\ji{x,y}$ is \emph{rational} if
$x,y \in \IQ$.

\begin{proposition}
  \label{prop:join-prime}
  Prime \osf{s} are exactly the join-prime elements of
  $\LjI$.
\end{proposition}

There are no \cjp elements in $\LjI$. Yet we have:
\begin{proposition}
  Every element of $\LjI$ is a join of rational
  \osf{s}. 
\end{proposition}
Meet-irreducible elements are easily characterized using duality; they
belong to the join-semilattice generated by the \jp elements. Using
duality, the following proposition is derived.
\begin{proposition}
  $\LjI$ is the Dedekind-MacNeille completion of the sublattice
  generated by the rational \osf{s}.
\end{proposition}

\paragraph{Join-irreducible elements of $\LId$.}
Let now $d \geq 3$ be fixed. The lattice $\LId$ is no more
distributive; we characterize therefore its \jirr elements.
We associate to  a vector $p \in \I^{d}$ the tuple
$\ji{p} \in \PrLI$ defined as follows:
\begin{align*}
  \ji{p} & := \Fam{\ji{p_{i},p_{j}} \mid (i,j) \in \cd}\,.
\end{align*}
\begin{proposition}
  The elements of the form $\ji{p} \in \PrLI$ are clopen and they are
  exactly the \jirr elements of $\LId$ (whenever $\ji{p} \neq
  \bot$). Every element of $\LId$ is the join of the \jirr elements
  below it.
\end{proposition}
As before $\LId$ is the Dedekind-MacNeille completion of its
sublattice generated by the \jirr elements. Yet, it is no longer true
that every element of $\LId$ is a join of elements of the form
$\ji{p}$ with all the $p_{i}$ rational
and therefore $\LId$ is not the \DMNc of its sublattice generated by
this kind of elements.

\medskip

Let us explain the significance of the previous observations.
For each vector $v \in \N^{d}$ there is an embedding $\iota_{v}$ of
the multinomial lattice $\L(v)$ (see \cite{BB,STA}) into $\LId$,
\begin{wrapfigure}[5]{r}
  {0.3\textwidth} 
  \onlyllncs{\vspace{-24pt}}
  \nollncs{\vspace{6pt}}
  \begin{tikzcd}
    \L(v) \arrow[rr,rightarrowtail,"\iota_v" description] && \LId
    \arrow[ll,bend right,"\ell_v" description]
    \arrow[ll,bend left,"\rho_v" description]
  \end{tikzcd} 
\end{wrapfigure}
as in the diagram on the right, where $\ell_{v}$ and $\rho_{v}$ are,
respectively, the left and right adjoint to $\iota_{v}$.  These
embeddings form a directed diagram whose colomit can be identified
with the sublattice of $\LId$ generated by the elements $\ji{p}$ with
all the $p_{i}$, $i \in \setof{d}$, rational.  The fact $\LId$ is not
the \DMNc of this sublattice means that, while we can still define
approximations of elements of $\LId$ in the multinomial lattices via
adjoints, these approximations do not converge to what they are meant
to approximate.
For example, we could define $\appr(f) := \ell_{v}(f)$ and yet have
$\bigvee_{v \in \N^{d}} \iota_{v}(\appr(f)) < f$.
On the other hand, it is possible to prove that every \mirr element is
an infinite join of \jirr elements arising from a rational
point. Therefore we can state:
\begin{proposition}
  \label{prop:last}
  Every element of $\LId$ is a meet of joins (and a join of meets) of
  elements in the sublattice of $\LId$ generated by the $\ji{p}$ such
  that $p_{i}$ is rational for each $i \in \setof{d}$.
\end{proposition}

Whether the last proposition 
is the key to use the lattices
$\LId$ as well as the multinomial lattices for higher dimensional
approximations in discrete geometry is an open problem that we shall
tackle in future research.

\section{Conclusions}
\label{sec:conclusions}

In this paper we have shown how to extend the lattice structure on a
set of discrete paths (known as a multinomial lattice, or \wBo, if the
words coding these paths are permutations) to a lattice structure on
the set of (images of) continuous paths from $\I$, the unit interval
of the reals, to the cube $\I^{d}$, for some $d \geq 2$.

By studying the structure of these lattices, called here $\LId$, we
have been able to identify an intrinsic difficulty in defining
discrete approximations of lines in dimensions $d \geq 3$ (a problem
that motivated us to develop this research).
This stems from the fact that $\LId$ is no longer (when $d \geq 3$) generated
by its sublattice of discrete paths as a Dedekind-Macneille
completion.
Proposition~\ref{prop:last} exactly describes how the lattice $\LId$
is generated from discrete paths and might be the key to use the
lattices $\LId$ as well as the multinomial lattices for defining
higher dimensional approximations of lines. We shall tackle this
problem in future research.

As a byproduct, our paper also pinpoints that various generalizations
of \Permutohedra crucially rely on the algebraic (but also logical)
notion of \msaq. Every such quantale yields an infinite family of
lattices indexed by positive integers.  While the definition of these
lattices becomes straightforward by means of the algebra, it turns out
that the elements of these lattices are (as far as observed up to now)
in bijective correspondence either with interesting combinatorial
objects (permutations, pseudo-permutations) or with geometric ones
(continuous paths, as seen in this paper). These intriguing
correspondences suggest the existence of a deep connection between
combinatorics/geometry and logic. Future research shall unravel these
phenomena. A first step, already under way for the Sugihara monoids on
a chain, shall 
systematically identify the combinatorial objects arising from a given
\msaq $Q$.

\nollncs{\newpage}
\bibliographystyle{abbrv}
\bibliography{biblio}

\end{document}